\documentclass[10pt]{amsart}
\usepackage{amssymb,amsmath,amsfonts,amsthm,mathrsfs,enumerate,cite,bm}
\usepackage{graphicx}
\usepackage{caption}
\usepackage{color}
\usepackage[usenames,dvipsnames]{xcolor}
\usepackage{tikz}
\usepackage{showexpl}
\usepackage{hyperref}
\usepackage[textsize=tiny]{todonotes}
\usepackage{cancel}
\usepackage[normalem]{ulem}
\def\d{{\,\rm d}}
\def\dist{\operatorname{dist}}
\def\bbC{\mathbb{C}}
\def\bbR{\mathbb{R}}
\def\bbS{\mathbb{S}}
\def\strain{\boldsymbol{\epsilon}}

\newtheorem{proposition}{Proposition}[section]
\newtheorem{theorem}[proposition]{Theorem}

\newtheorem{lemma}[proposition]{Lemma}
\theoremstyle{definition}

\newtheorem{remark}[proposition]{Remark}

\numberwithin{equation}{section}
\newcommand{\weakto}{\rightharpoonup}
\newcommand{\Om}{\Omega}
\newcommand{\R}{\mathbb{R}}
\newcommand{\tr}{\operatorname{tr}}

\title[]
{Linearization in magnetoelasticity}

\author[S. Almi]
{Stefano Almi}
\address[S. Almi ]{University of Naples Federico II, Department of Mathematics and Applications ``Renato Caccioppoli'', via Cintia, Monte S. Angelo, 80126 Naples, Italy.}
\email{stefano.almi@unina.it}

\author[M. Kru\v{z}\'{i}k]
{Martin Kru\v{z}\'{i}k}
\address[M. Kru\v{z}\'{i}k]{Czech Academy of Sciences, Institute of Information Theory and Automation,
Pod vod\'arenskou v\v{e}\v z\'\i\ 4, 182 00, Prague 8, Czechia \& Faculty of Civil Engineering, Czech Technical University, Th\'{a}kurova 7, 166 29, Prague 6, Czechia}
\email{kruzik@utia.cas.cz}

\author[A. Molchanova] {Anastasia Molchanova} 
\address[A. Molchanova]{University of Vienna, Faculty of Mathematics, Oskar-Morgenstern-Platz 1, 1090 Vienna, Austria \& TU Wien, Institute of Analysis and Scientific Computing, Wiedner Hauptstra\ss e 8--10, 1040 Wien}
\email{anastasia.molchanova@tuwien.ac.at}

\begin{document}

\begin{abstract}
    Starting from a model of nonlinear magnetoelasticity where magnetization is defined in the Eulerian configuration while elastic deformation is in the Lagrangian one,  we rigorously derive a linearized model that coincides with the standard one that already appeared in the literature, see, e.g.,  \cite{DeSimJam2002} and where the zero-stress strain is quadratic in the magnetization.
    The relation of the nonlinear and linear model is stated in terms of the $\Gamma$-convergence and convergence of minimizers.
\end{abstract}


\maketitle

\section{Introduction}

Magnetoelasticity is a characteristic observed in certain solids wherein there is a significant interdependence between their mechanical and magnetic properties. This coupling manifests as reversible mechanical deformations that can be induced by applying an external magnetic field. This phenomenon holds considerable practical value, especially in the design of sensors, actuators, and various innovative functional-material devices,
  see e.g.\  \cite{GaoZengPengSchuai2002}. 

The origin of magnetoelasticity can be traced to the intricate interplay between the crystallographic structure of the material, where different crystals exhibit distinct easy axes of magnetization, and magnetic   domains~\cite{HubSha1998}.  In the absence of external magnetic fields, these magnetic domains align to minimize long-range dipolar effects, leading to a generally minimal or even negligible magnetization of the medium. Upon the application of an external magnetic field, the magnetic domains undergo a reorientation towards it. This reorientation leads to the emergence of a macroscopic deformation since magnetizations are linked to specific stress-free reference strains. As the magnetic field's intensity increases, an increasing number of magnetic domains align themselves so that their principal axes of anisotropy align with the magnetic field in each region, eventually reaching saturation. For a deeper exploration of the foundational aspects of magnetoelasticity, refer to \cite{Bro1966,DorOgd2014,HubSha1998}.

Magnetoelasticity is a dynamic area of mathematical research which reflects modeling in linear and nonlinear regimes \cite{DeSimJam2002,WanSte2012}, analysis in the static and quasistatic setting  \cite{BarHenMor2017,BreDavKru2023,Bre2021} while also being
coupled with thermal and other phenomena, e.g.~ \cite{RouTom2013}. 
Moreover, dimension-reduction problems covering lower-dimensional structures in various scaling regimes were recently as well analyzed; for instance, see \cite{Bre2021} or \cite{DaKrPiSt21}. 

The aim of our contribution is to establish a relationship between nonlinear and linear models of magnetoelasticity in the static case. To our best knowledge, it is the first situation where this program is undertaken in the combined Eulerian and Lagrangian description.

Let $\Omega$ be a (connected)  bounded domain in $\bbR^d$ $(d\geq 2)$ with Lipschitz boundary~$\partial \Om$, representing a reference configuration occupied by an elastic body. We fix $\Gamma$ open subset of $\partial\Om$ such that $\mathcal{H}^{d-1} (\Gamma)>0$, where $\mathcal{H}^{d-1}$ denotes the $(d-1)$-dimensional Hausdorff measure. We consider the magnetoelastic energy of the following form \cite{Bro1966,JamKin1993,BreDavKru2023,RybLus2005}: 
\begin{equation}
\label{e:energy-lin}
    \mathcal{G}(y,m):= \int_{\Omega} W(\nabla y, m \circ y) \d x + \frac{1}{2}\int_{y(\Omega)} | \nabla m|^{2} \d z + \frac{\mu_0}{2} \int_{\bbR^d} |\nabla v_m|^2 \d z
\end{equation} 
for a deformation $y \in W^{1,p} (\Omega; \bbR^{d})$ 
with $p> d$, 
$\det \nabla y > 0$ a.e., $y$ is injective a.e.~in $\Omega$, and $y=y_{0}$ on $\Gamma$ where $y_{0} \in  W^{1,p} (\Omega; \bbR^{d})$ is a given boundary datum,
and a magnetization $m \in W^{1,2}(y(\Omega)\setminus y(\partial \Omega); \bbR^{d})$. 
In \eqref{e:energy-lin}, the first term represents the elastic energy, the second term is the exchange energy, and
the last term stands for magnetostatic energy;~$\mu_{0}$ is the permittivity of void and~$v_m$ is the magnetostatic potential generated by~$m$. 
In particular,~$v_m$ is a solution to the Maxwell equation
$\nabla \cdot(-\mu_0 \nabla v_m + \chi_{y(\Omega)} m) =0$ in $\bbR^d$ with~$\chi_{y(\Omega)}$ the characteristic function of~$y(\Omega)$.  Furthermore, we can assume that the specimen temperature is fixed and therefore the Heisenberg normalized constraint 
\begin{align}\label{heisenberg}
|m \circ y|\, \det \nabla y = 1 \text{ a.e.~in } \Omega
\end{align}
holds for all deformations
\cite{Bro1966,JamKin1993}.

Let us point out some peculiarities caused by the mixed Eulerian--Lagrangian formulation. 
First, to correctly define the composition $m \circ y$, we must identify $y$ with its continuous representative, due to Sobolev embeddings theorems \cite[Section 1.4.5]{Maz2011}.
Second, as it is conventional in magnetoelasticity \cite{RybLus2005,BreDavKru2023,Bresciani-Kruzik}, we define $m \in W^{1,2}(y(\Omega) \setminus y(\partial \Omega); \bbR^{d})$, since $y(\Omega) \setminus y(\partial \Omega)$ is an open set 
and it differs form $y(\Omega)$ just on a set of measure zero, i.e.~$|y(\Omega)| = |y(\Omega) \setminus y(\partial \Omega)|$,
see for example \cite{FonGan1995} or \cite[Lemma 2.1]{BreDavKru2023}.

  As in \cite{EthMie2016, Yavari},  we suppose the deformation gradient $F$ can be multiplicatively  decomposed into magnetic and elastic parts $F = F_{m}F_{el}$
with the magnetic part $F_{m}$ being volume-preserving, i.e., $\det F_m=1$.
In our setting, we assume 
$F_{m}:=\exp(-E(M))$, i.e., $F^{-1}_{m}= \exp (E(M))$, where $E(M)$ is defined as 
\[
    E(M): =-M\otimes M+\frac{1}{d} I.
\]
Here, $M\in\R^d$ is a placeholder for the Lagrangian magnetization whose length $|M|$ satisfies $M\cdot M=|M|^2=1$, and $I$ is the identity matrix in $\bbR^{d\times d}$. 
Note that $\tr E(M)=0$, so that, indeed, $\det F_m^{-1}=1$.  In view of the Heisenberg constraint \eqref{heisenberg}, 
\begin{align}\label{Heisenberg}
    M(x)=m(y(x))\det\nabla y(x)
\end{align}  
satisfies $|M(x)|=1$ for  $x\in\Omega$. 
Consequently,  
using the change-of-variables formula for Sobolev mappings \cite{Haj1993}  (considering that 
$y$ is a.e.~injective and satisfies the Lusin $(N)$-condition)
for $\omega\subset\Omega$
it holds 
\[
    \int_{\omega}M(x)\, \d x=\int_{y(\omega)}m(z)\, \d z. 
\]
 Therefore,  $M$ and $m$ transform as volume densities in the reference (Lagrangian) and the deformed (Eulerian) configurations, respectively. 

We define  $e\colon\R^{d\times d}\times\R^d\to\R^{d\times d}$  as
\[
    e(F,m)= -(\det F)^2 m\otimes m+\frac{1}{d} I.
\]
Whenever $m$ and $M$ are two magnetic fields related by \eqref{Heisenberg},  then $e(\nabla y,m\circ y)=E(M)$ in $\Om$.

In the case $F = I$, we simplify the notation by setting $e(m):=e(I,m) =- m \otimes m + \frac{1}{d} I$. 
Next, we assume that the elastic energy density is of form
\[
    W(F,m) = \Phi\left(\exp(e(F,m))F\right)=\Phi(F_{el}). 
\] 
We recall that the matrix exponential $\exp (A)= \sum_{k=0}^\infty A^k/k!$ is defined for every $A\in\R^{d\times d}$, setting also 
$A^0=I$. Here, $\Phi$ is a nonnegative function that is minimized precisely at the set of proper rotations. This form of the energy density is inspired by energy expressions in the following papers on nematic elastomers \cite{Agostiniani-DeSimone_2011,Agostiniani-DeSimone_2017, Agostiniani-DeSimone_2020,DeSimone-Teresi_2009}. 

In the \textit{small-strain regime}, we set 
$y_\varepsilon = id + \varepsilon u$ and assume that the influence of the magnetization on the elastic deformation of the specimen vanishes along $\varepsilon\to 0$ in the sense that $F_{m,\varepsilon}=\exp(-  \varepsilon E(M))$ and $\lim_{\varepsilon\to 0} F_{m,\varepsilon}=I$. 
Then, we consider the rescaled energy 
\begin{equation}
\label{e:energy-lin-2}
\begin{aligned}
    \mathcal{G}_\varepsilon(u, m) &:= \frac{1}{\varepsilon^{2}} \int_{\Omega} W(Id + \varepsilon \nabla u (x), m (x + \varepsilon u(x)))\d x + \frac{1}{2}\int_{y_\varepsilon(\Omega)} | \nabla m (z)|^{2}\d z + \frac{\mu_0}{2} \int_{\bbR^d} |\nabla v_m (z)|^2 \d z\\
   &=  \frac{1}{\varepsilon^{2}} \int_{\Omega}\Phi(\exp(\varepsilon E(M))(I+\varepsilon\nabla u(x)))\, \d x +\frac{1}{2}\int_{y_\varepsilon(\Omega)} | \nabla m (z)|^{2}\d z + \frac{\mu_0}{2} \int_{\bbR^d} |\nabla v_m (z)|^2 \d z.
    \end{aligned}
\end{equation}

We formally get the expression for the elastic strain 
\[
    F_{el,\varepsilon}=F^{-1}_{m,\varepsilon}\nabla y_\varepsilon= \operatorname{exp}(\varepsilon E(M))\cdot(I + \varepsilon \nabla u) \approx I + \varepsilon (\nabla u + e(m)).
\] 
Moreover, a formal Taylor expansion shows that the limit energy is, as in \cite{DeSimJam2002}, of the form:
\begin{equation}\label{e:energy-lin-0}
\mathcal{G}_{0}(u, m) :=  \frac{1}{{2}} \int_{\Omega} \bbC (\strain(u) + e (m)) : (\strain(u) + e (m)) \d x + \frac{1}{2}\int_{\Omega} | \nabla m|^{2}\d x + \frac{\mu_0}{2} \int_{\bbR^d} |\nabla v_m|^2 \d z,
\end{equation}
where 
$e(u) := \frac{1}{2} (\nabla u^T + \nabla u)$ is the symmetric part of the displacement gradient and
$\bbC:=D^2 \Phi(I)$ is the elasticity tensor. 
Note that \eqref{e:energy-lin-0} suggests a decomposition of the total strain $\strain(u)$ into the magnetic part $-e(m)=m\otimes m-I/d$ and the elastic part $\strain(u)+e(m)$.  
This decomposition has already been used, e.g., in  \cite{DeSimJam2002, KrStZa2015,mudivarthi2010 }. 
The aim of this contribution is to derive \eqref{e:energy-lin-0} rigorously by means of $\Gamma$-convergence.

We use the following assumptions for $p>d$: 
\begin{enumerate}
    \item[(a)] \label{hyp:a} 
        $\Phi$ is frame-indifferent;
    \item[(b)] \label{hyp:b} 
        $\Phi$ belongs to $C^{2}$ in some neighbourhood of $SO(d)$;
    \item[(c)] \label{hyp:c} 
        $\Phi(F) = 0$ if $F\in SO(d)$;
    \item[(d)] \label{hyp:d} 
        $\Phi(F) \geq g_p (\dist(F,SO(d)))$ for some $p$, where
        \begin{equation}\label{def:gp}
            g_p(t):=
            \begin{cases}
                \frac{t^2}{2}, & \text{if } 0\leq t\leq 1,\\
                \frac{t^p}{p} + \frac{1}{2} - \frac{1}{p}, & \text{if } t>1.
            \end{cases}
        \end{equation}
    \item[(e)] \label{hyp:e} 
        there exist $a>1$ and~$C>0$ such that
        \begin{displaymath}
                 \Phi (F) \geq C \bigg(\frac{1}{(\det F)^{a}} - 1 \bigg) \qquad \text{for every $F \in {\bbR^{d\times d}}$.} 
        \end{displaymath}
\end{enumerate}

Let us notice that $g_{p}$ satisfies the following properties:
$g_{p}$ is convex, 
\begin{align}
\label{e:g-properties-1}
    g_p (t) \leq \frac{1}{p} \min \{t^p, t^2\}\,, \qquad 
g_{p}(s+t) \leq C (g_{p}(s) + t^2)
\end{align} 
for some constant $C>0$ depending on $p$, and for any $C_{1} >0$ there exists a constant $C_{2}>0$, depending only on~$p$ and on~$C_{1}$, such that 
\begin{equation}
    \label{e:g-properties-2}
    g_{p}(C_1 t) \leq C_2 g_{p}(t)\,.
    \end{equation}

Frame-indifference \hyperref[hyp:a]{(a)}  implies that $W(F,m)=W(F,-m)=W(QF,Qm)$ for  all $F\in\R^{d\times d}$ all $Q\in\textrm{SO}(d)$, and all $m\in\R^d$. 
Moreover, for every $  F   \in\R^{d\times d}$ 
\begin{equation}
    D^2\Phi(I)[  F,F  ] = D^2\Phi(I)[  F_{sym},F_{sym}  ].
\end{equation}
Together with \hyperref[hyp:d]{(d)}, it implies that 
$D^2\Phi(I)[  F,F  ] = 0$ if $  F=-F^\top  $ 
and 
\begin{equation}
    D^2\Phi(I)[  F_{sym},F_{sym}  ] \geq |  F_{sym}  |^2 .
\end{equation}
Above   $F_{sym}$   stands for the symmetric part of   $F$.  
While \hyperref[hyp:b]{(b)} together with $\Phi(I) = D \Phi(I) = 0$ leads to 
\begin{equation}\label{neq:quadrat}
    \Phi(I+F) \geq \frac{1}{2} D^2 \Phi(I)[F,F] - \eta(|F|)|F|^2,
\end{equation}
where $\eta$ is an increasing nonnegative function such that $\eta(t) \leq C_{R} t $ for $|t| \leq R$ and $R >0$.

In what follows, the work of a mechanical load is modeled by a continuous linear functional $\mathscr{L}\colon L^{2}(\Omega;\bbR^d)\to \bbR$ and the magnetic loading  is described by a  functional $\mathscr{M}\colon L^{2}(\Omega;\bbR^d)\times L^2(\bbR^d;\bbR^d)\to \bbR$
\begin{equation*}
    \mathscr{L}(y):= \int_{\Omega} f\cdot y \d x,
    \qquad \text{and} \qquad
    \mathscr{M}(y,m):= \int_{y(\Omega)} h\cdot m \d z
\end{equation*}
where $f\in L^{2}(\Omega;\bbR^d)$ is a body force and 
$h\in L^2(\bbR^d;\bbR^d)$  represents an external magnetic field, while~$\mathscr{M}$ is called the Zeeman energy.

If $y \in W^{1, p}(\Omega;\bbR^d)$ represents the deformation of the elastic body, the stable equilibria of the elastic body are obtained by minimizing the functional
\[
    \int_{\Omega} W(\nabla y,m\circ y) \d x-\mathscr{L}(y)-\mathscr{M}(y,m) ,
\]
under the prescribed boundary conditions. 
The area of interest of this paper is
the case where the load has the form~$\varepsilon \mathscr{L}$, and 
to study the behavior of the solution as~$\varepsilon$ tends to zero. We write
\[
     y_{\varepsilon} =id  +\varepsilon u
\]
and
assume the Dirichlet boundary condition of the form
\[
     y_{\varepsilon} = id+\varepsilon w \quad \mathcal{H}^{d-1} \text {-a.e. on } \Gamma , 
\]
with a given function $w \in W^{2, \infty}(\Omega; \bbR^d)$. 
Denoting a set of functions satisfying boundary conditions as
\[
    \begin{aligned}
        W^{1,p}_w(\Omega; \bbR^d)& := & \left\{u\in W^{1,p}(\Omega; \bbR^d): u=w \quad \mathcal{H}^{d-1} \text {-a.e. on } \Gamma \right\}\!, \\
        W^{2,\infty}_w(\Omega; \bbR^d)& := & \left\{u\in W^{2,\infty}(\Omega; \bbR^d): u=w \quad \mathcal{H}^{d-1} \text {-a.e. on } \Gamma \right\}\!,
        \\
    \end{aligned}
\]
the corresponding minimum problem for $(u,m)$ becomes
\begin{align*}
    \min \biggl\{\int_{\Omega} W(I+\varepsilon \nabla u (x), m(x + \varepsilon u(x))) \d x-\varepsilon \mathscr{L}(\varepsilon u) & - \mathscr{M}(id+\varepsilon u,m) : \\
    & (u,m) \in W^{1, p}_w (\Omega; \bbR^d) \times W^{1,2}( y_{\varepsilon}(\Omega) \setminus y_{\varepsilon} (\partial \Om)  ; \bbR^d) \biggr\}\!.
\end{align*}
The class of admissible deformations is given by
\begin{equation}
    \mathcal{Y}:=\left\{y\in W^{1,p}(\Omega;\bbR^d): \det \nabla y >0 \text{ a.e. in } \Omega, \: y \text{ is a.e.~injective in } \Omega \right\}\!.
\end{equation}
For $w\in W^{2,\infty} (\Om; \R^{d})$ and $\varepsilon >0$, define the class of admissible states as
\begin{align*}
    \mathcal{A}^w_\varepsilon :=\bigl\{(u,m)& \in W^{1,2}_w(\Omega;\bbR^d) \times L^2(\bbR^d;\bbR^d): y_\varepsilon: = id+\varepsilon u \in \mathcal{Y},  \\
    & m \in W^{1,2}( y_\varepsilon(\Omega) \setminus y_{\varepsilon} (\partial \Om)  ; \bbR^d), \: |m \circ y_\varepsilon|\, \det \nabla y_\varepsilon = 1 \text{ a.e.~in } \Omega\bigr\},
\end{align*}
and
\begin{equation*}
    \mathcal{A}^w_0 := W^{1,2}_w(\Omega;\bbR^d) \times  W^{1,2}(\Omega; \bbS^{d-1}).
\end{equation*}

The main result of this work is the following $\Gamma$-convergence statement.

\begin{theorem}\label{thm:gamma-convergence}
    Assume that $W\colon \bbR^{d\times d} \times \bbR^d \rightarrow[0, \infty]$ satisfies conditions \hyperref[hyp:a]{(a)}--\hyperref[hyp:e]{(e)} for some $p > d $, and let $w \in W^{2, \infty}(\Omega; \bbR^d)$. For every $\varepsilon_j \to 0$ we have that
    \[
        \mathcal{G}_{\varepsilon_j} \stackrel{\Gamma}{\longrightarrow} \mathcal{G}_0, \quad \text { as } j \to \infty
    \]
    in the weak topology of $W^{1, 2}(\Omega; \bbR^d) \times L^{2}(\bbR^d;\bbR^d)$.
\end{theorem}

\begin{remark}
    In view of the compactness properties stated in Proposition~\ref{prop:compactness} and in order to state the $\Gamma$-convergence in $W^{1, 2}(\Omega; \bbR^d) \times L^{2}(\bbR^d;\bbR^d)$, In Theorem~\ref{thm:gamma-convergence} we identify~$m_{\varepsilon}$ with $\chi_{y_{\varepsilon} (\Om)} m_{\varepsilon}$  for $(u_{\varepsilon}, m_{\varepsilon}) \in \mathcal{A}^{w}_{\varepsilon}$ and $y_{\varepsilon} = id + \varepsilon u_{\varepsilon}$. Similarly, in the limit as $\varepsilon \to 0$ we identify the magnetization~$m$ with~$\chi_{\Om} m$.
\end{remark}

Theorem~\ref{thm:gamma-convergence} is completed by the following theorem, stating the convergence of minima and minimizers. Here, we consider loading terms as well. In particular, for every $\varepsilon>0$ and every $(u, m) \in \mathcal{A}^{w}_{\varepsilon}$ we set
\[
    \mathcal{F}_\varepsilon (u,m):=\mathcal{G}_\varepsilon (u,m) - \mathscr{L}(u) - \mathscr{M}(id+\varepsilon u,m).
\]

\begin{theorem}\label{thm:linearization}
    Under the hypotheses of Theorem~\ref{thm:gamma-convergence}, for every $\varepsilon>0$ define $s_{\varepsilon}:=\inf\{\mathcal{F}_\varepsilon (u,m) : (u,m) \in \mathcal{A}^w_\varepsilon\}$ and let 
    $(u_{\varepsilon}, m_{\varepsilon}) \in \mathcal{A}^w_\varepsilon $ be a sequence such that
    \begin{equation}
        \label{e:thm:lin}
        \mathcal{F}_\varepsilon (u_{\varepsilon},m_{\varepsilon})
        =s_{\varepsilon}+o(1) \quad \text{ as } \varepsilon \to 0.
    \end{equation}
    Then, $u_{\varepsilon} \weakto u_0$ weakly in $W^{1, 2}(\Omega;\bbR^d)$, $\chi_{y_{\varepsilon} (\Om)} m_{\varepsilon}\to \chi_{\Omega} m_0$ in $L^{2}(\bbR^d;\bbR^d)$, and $\chi_{y_\varepsilon(\Omega)}\nabla m_{\varepsilon}\weakto \chi_\Omega \nabla m_0$ weakly in $L^{2}(\bbR^d;\bbR^{d\times d})$, where $(u_0,m_0)\in \mathcal{A}_0^w$ is the solution to the minimum problem
    \[
        s_0:=\min\bigl\{\mathcal{F}_0(u,m): (u,m) \in \mathcal{A}^{w}_0\bigr\},
    \]
    where $\mathcal{F}_0(u,m):=\mathcal{G}_0(u,m)-\mathscr{L}(u) - \mathscr{M}(id,m)$ and $\mathcal{G}_0$ is defined by \eqref{e:energy-lin-0}. 
    Moreover, $s_{\varepsilon} \rightarrow s_0$.
\end{theorem}

\begin{theorem}\label{thm:convergence_of_minimizers}
    Under the hypotheses of Theorems~\ref{thm:gamma-convergence} and~\ref{thm:linearization}, let $(u_{\varepsilon}, m_{\varepsilon}) \in \mathcal{A}^{w}_{\varepsilon}$ satisfy~\eqref{e:thm:lin}. Then, up to a subsequence, 
    $(u_{\varepsilon},  \chi_{y_{\varepsilon} (\Om)} m_\varepsilon)  \to (u_0,  \chi_{\Om}  m_0)$ strongly in 
    $W^{1, 2}(\Omega; \bbR^d) \times L^{2}(\bbR^d;\bbR^d)$.
\end{theorem}

The proof of Theorems~\ref{thm:gamma-convergence}--\ref{thm:convergence_of_minimizers} is postponed to Section~\ref{s:conv}. Compactness issues are discussed in Section~\ref{s:compactness}.

\begin{remark}[Discussion for the case $p \leq d$]
    Let us notice that the $p$-integrability of deformations' gradients plays an essential role only for the compactness.
    The result of this paper can be extended for the case $p=d$ using the means of quasiconformal analysis, see details in, e.g.,~\cite{HenKos2014}. 
    First, an admissible deformation $y\in \mathcal{Y}$ is continuous and differentiable a.e., and a continuous representative satisfies the Lusin $(N)$-condition.
    Thus, the composition $m\circ y$ is well defined, and the change-of-variables formula is valid for the chosen representative. 
    Second, instead of $y(\Omega) \setminus y(\partial \Omega)$ one needs to use the topological image, 
    see e.g.~\cite{MulSpe1995,BarHenMor2017},
    to define 
    $m$
    properly.
    Finally, there is no uniform convergence $y_\varepsilon \to y$, only locally uniform convergence.
    Therefore, the proof of compactness in Proposition~\ref{prop:compactness} needs more refined estimates.
    
    There are several options in the literature, which offer ways to deal with the case  $p > d-1$.
    For instance, one may use the approach prohibiting cavitations as in  M\"uller--Spector \cite{MulSpe1995}, Barchiesi--Henao--Mora-Coral \cite{BarHenMor2017}, and Bresciani \cite{Bre2023}.
    Although a stronger coercivity condition is needed to guarantee weak convergence of determinants. 
    Another option, is to consider deformations with finite surface energy as in \cite{BarHenMor2017}, (INV)-condition or 
    weak limits of homeomorphisms \cite{BouHenMol2020,DolHenMal2021}.
    Admissible deformations in
    the case $p=d-1$ may be considered within the limits of homeomorphisms \cite{DolHenMol2024} or mappings with finite surface energy \cite{BarHenMorCorRod2023,BarHenMorCorRod2024}.
\end{remark}

\section{Auxiliary results}

The homogeneous Sobolev space  
$L^{1,2}(\bbR^d)$
is defined
as the space of $\varphi \in L^2_{loc}(\bbR^d)$ such that $\nabla \varphi \in L^2(\bbR^d; \bbR^d)$.
We equip the space $L^{1,2}(\bbR^d)$ with a seminorm $\| \nabla \varphi\|_{L^2(\bbR^d; \bbR^d)}$. 
It can be as well supplied with a norm 
$\|\varphi\|_{L^{1,2}(\bbR^d)} :=\| \nabla \varphi\|_{L^2(\bbR^d; \bbR^d)} + \| \varphi\|_{L^2(\omega)}$, 
where $\omega\subset \bbR^{d}$ is arbitrary bounded open nonempty set.
For more details, the interested reader is referred to \cite[Chapter 1]{Maz2011}.
The magnetic potential $v_m \in L^{1,2}(\bbR^d)$ generated by magnetization $m$ exists, unique up to a constant, and continuously depends on $m$.

\begin{proposition}
\label{prop:magn}
    For every $f \in L^{2}(\bbR^d;\bbR^d)$ there exists $v_{f} \in L^{1,2}(\bbR^d)$ which is unique up to additive constants, such that 
    \begin{equation*}
        \int_{\bbR^d}(-\nabla v_f + f)\cdot \nabla \varphi \d z = 0 \qquad \text{for all } \varphi\in L^{1,2}(\bbR^d)
    \end{equation*}
    and
    \begin{equation*}
        \| \nabla v_f\|_{L^2} \leq \|f\|_{L^2}.
    \end{equation*}
    Moreover, if $f_{j} \to f$ strongly in $L^{2}$, then
        $\nabla v_{f_j}\to \nabla v_{f}$ strongly in  $L^2(\bbR^d, \bbR^{d})$.
\end{proposition}
\begin{proof}
    This proposition is classical, see for example \cite[Proposition 8.8 \& Theorem 8.9]{BarHenMor2017} or \cite[Proof of Theorem 4.1]{RybLus2005}.
\end{proof}

Let us now recall sufficient conditions for a deformation to be injective if the gradient of the displacement is small enough in the operator norm, i.e., $|B|_{O} = \sup_{v\in \bbR^d} \frac{|Bv|}{|v|}$.  
\begin{theorem}[\hspace{-1pt}{\cite[Theorem 5.5-1]{Cia1988}}] \label{thm:injectivity}
Let $\Omega \subset \bbR^{d}$ be an open set.
\begin{enumerate}
    \item[(1)] 
        Let $y=id+u \colon \Omega \to \bbR^{d}$ be differentiable in $x\in \Omega$.
        Then
        \[
            |\nabla u(x)|_{O} <1 \Rightarrow \det \nabla y(x) >0.
        \] 
    \item[(2)] 
        Let $\Omega \subset \bbR^{d}$ be a bounded domain with Lipschitz boundary. Then there exists a constant $c(\Omega)$ such that any 
        $y=id+u \in C^{1}(\bar{\Omega};\bbR^d)$ with 
        $\sup_{x\in \bar{\Omega}}|\nabla u(x)|_O<c(\Omega) $
        is injective.
\end{enumerate}
\end{theorem}

To establish compactness, we need the following geometric rigidity, proven in \cite[Theorem~3.1]{FriJamMul2002} for $p=2$ and in \cite[Section 2.4]{Conti-Schweizer} for $p \in (1, +\infty)$.

\begin{theorem}
\label{lem:gp-rigidity}
    There exists a constant $C = C(\Omega, p) > 0$ such that for every $v \in W^{1,p}(\Omega; \R^{d})$ there exists a constant rotation $R \in SO(d)$ satisfying
    \begin{equation}\label{eq:gp-rigidity}
        \int_\Omega |\nabla v - R|^{p} \d x \leq C \int_\Omega g_p({\dist}(\nabla v, SO(d))) \d x,
    \end{equation}
    where $g_p$ is defined by \eqref{def:gp}. In particular, we may choose $R$ such that
    \begin{equation}
    \label{e:rotation-choice}
        \left| R - \frac{1}{\mathcal{L}^{d} (\Omega)} \int_{\Omega} \nabla v \, \d x \right| = {\dist} \left(\frac{1}{\mathcal{L}^{d} (\Omega)} \int_{\Omega} \nabla v \, \d x , SO(d) \right)\!.
    \end{equation}
\end{theorem}

\begin{proof}     
    It is enough to resort to the classical rigidity estimates for $p \in (1, +\infty)$ and notice that
    \begin{equation*}
        \frac{1}{2p} (t^p + t^2)\leq \frac{1}{p} \max \{t^{p}, t^{2}\} \leq g_{p}(t) \leq \frac{1}{2} \max \{t^{p}, t^{2}\} \leq \frac{1}{2} (t^p + t^2).
    \end{equation*}
    The particular choice of~$R$ as in~\eqref{e:rotation-choice} can be ensured arguing, e.g., as in~\cite[Theorem~3.1]{ACFS-23}.
\end{proof}

\section{Compactness}
\label{s:compactness}

This section is devoted to the proof of the following compactness result for sequences with bounded energy $\mathcal{G}_{\varepsilon}$.

\begin{proposition}[Compactness]
\label{prop:compactness}
    Under the assumptions of Theorem~\ref{thm:linearization}, let $(u_{\varepsilon}, m_{\varepsilon}) \in \mathcal{A}^{w}_{\varepsilon}$ be such that 
    \begin{equation}
    \label{e:energy-bound}
        \sup_{\varepsilon>0} \, \mathcal{G}_{\varepsilon} (u_{\varepsilon}, m_{\varepsilon})<+\infty.
    \end{equation}
    Then, there   exists  $(u_0,m_0)\in W^{1,2}_w(\Omega;\bbR^d) \times  W^{1,2}(\Omega; \bbS^{d-1})$ such that, up to a subsequence,
    \[
        \begin{aligned}
            & y_{\varepsilon}:= id + \varepsilon u_\varepsilon \to id && \text{ strongly in } W^{1, p}(\Omega;\bbR^d), \\
            & u_\varepsilon \weakto u_0 && \text{ weakly in } W^{1,2}(\Omega;\bbR^d),\\
            & \chi_{{y}_\varepsilon(\Omega)}{m}_\varepsilon \to \chi_{\Omega} m_0 && \text{ strongly in } L^2(\bbR^d;\bbR^d),\\
            & \chi_{{y}_\varepsilon(\Omega)}\nabla{m}_\varepsilon \weakto \chi_{\Omega} \nabla m_0 && \text{ weakly in } L^2(\bbR^d;\bbR^{d\times d}),\\
            & m_{\varepsilon} \circ y_{\varepsilon} \to m_0 &&\text{ strongly in }L^{1} (\Omega; \bbR^{d}),\\
            &  m_\varepsilon \circ {y}_{\varepsilon} \det \nabla y_{\varepsilon} \to m_0 && \text{ strongly in } L^r(\Omega;\bbR^d) \text{ for every } 1\leq r < \infty.
        \end{aligned}
    \]
\end{proposition}

Before proving Proposition~\ref{prop:compactness}, we show the following lemmas, which provide boundedness of the displacement variable in terms of the energy~$\mathcal{G}_{\varepsilon}$.

\begin{lemma}
    \label{l:epsilon-well}
    Under the assumptions of Theorem \ref{thm:gamma-convergence}, there exist positive constants~$C_1=C_1(\Om, d, p)>0$ and $C_2=C_2(\Om, d, p)>0$ such that for every $\varepsilon \in (0, 1)$ and every $(u, m) \in \mathcal{A}^{w}_{\varepsilon}$ it holds
    \begin{equation}
    \label{eq:epsilon-well-energy}
        \varepsilon^{2}\mathcal{G}_{\varepsilon} (u, m) \geq C_1 \int_{\Om} g_{p} (\dist (\nabla y_{\varepsilon}, SO(d))\, \d x - C_2 \varepsilon^{2},
    \end{equation}
    for $y_{\varepsilon} = id + \varepsilon u$.
\end{lemma}

\begin{proof}
Let $(u, m) \in \mathcal{A}^{w}_{\varepsilon}$ and $\varepsilon \in (0,1)$. By assumption \hyperref[hyp:d]{(d)} we can estimate
\begin{equation}
\label{eq:G-epsilon-well}
    \varepsilon^{2}\mathcal{G}_{\varepsilon} (u, m) \geq C \int_{\Om} g_{p} (\dist \big(\exp (\varepsilon e(\nabla y_{\varepsilon}, m \circ y_{\varepsilon}) \nabla y_{\varepsilon}; SO(d) \big) \, \d x .
\end{equation}
By definition of $e(\nabla y_{\varepsilon}, m \circ y_{\varepsilon})$, by the properties of the exponential map, and by the fact that $\det{\nabla y_{\varepsilon}}| m \circ y_{\varepsilon}| = 1 $ in $\Om$, we have that for $x \in \Om$ and $R \in SO(d)$
\begin{align*}
    \dist & (\nabla y_{\varepsilon}; \exp(- \varepsilon e(\nabla y_{\varepsilon}, m \circ y_{\varepsilon})) SO(d)) \leq | \nabla y_{\varepsilon} - \exp(- \varepsilon e(\nabla y_{\varepsilon}, m \circ y_{\varepsilon}))R | 
    \\
    &
    \leq | \exp (- \varepsilon  e(\nabla y_{\varepsilon}, m \circ y_{\varepsilon})) | \, | \exp(\varepsilon  e(\nabla y_{\varepsilon}, m \circ y_{\varepsilon})) \nabla y_{\varepsilon} - R| \leq C | \exp(\varepsilon  e(\nabla y_{\varepsilon}, m \circ y_{\varepsilon})) \nabla y_{\varepsilon} - R|.
\end{align*}
This implies that for a.e.~$x \in \Om$
\begin{align*}
    \dist (\nabla y_{\varepsilon}; \exp(- \varepsilon e(\nabla y_{\varepsilon}, m \circ y_{\varepsilon})) SO(d)) \leq C \,\dist (\exp(\varepsilon e(\nabla y_{\varepsilon}, m \circ y_{\varepsilon}) \nabla y_{\varepsilon}; SO(d))) \, . 
\end{align*}
 Again by the properties of the exponential, we further estimate 
\begin{align}
\label{e:new}
    \dist (\nabla y_\varepsilon ; SO(d)) \leq C \, \dist (\nabla y_\varepsilon ; \exp(-\varepsilon e(\nabla y_\varepsilon, m \circ y_\varepsilon) SO(d)) + C \varepsilon \qquad \text{a.e.~in~$\Om$.}
\end{align}
Combining~\eqref{eq:G-epsilon-well}--\eqref{e:new} and the properties of $g_{p}$~\eqref{e:g-properties-1}--\eqref{e:g-properties-2}, we get~\eqref{eq:epsilon-well-energy}.
\end{proof}

\begin{lemma}
\label{prop:compactness_u}
    Under the assumptions of Theorem~\ref{thm:gamma-convergence}, there exists a positive constant $C = C(\Omega, d, p)>0$ such that for every $\varepsilon>0$, every $(u, m)  \in \mathcal{A}^{w}_{\varepsilon}$, and every $\varepsilon>0$ the following estimates hold:
    \begin{align}
        & \label{eq:rigidity-1}\int_{\Omega} | \nabla u|^{2}\d x \leq C \Bigg(\mathcal{G}_{\varepsilon} (u,m) + \int_{ \Gamma } | w |^{2}\d \mathcal{H}^{d-1}  +1 \bigg),\\
        & \label{eq:rigidity-2} \int_{\Omega} | \varepsilon \nabla u|^{p}\d x \leq C \varepsilon^{2} \bigg(\mathcal{G}_{\varepsilon} (u,m) + \int_{ \Gamma } | w|^2 \d \mathcal{H}^{d-1} +1  \bigg).
    \end{align}
\end{lemma}

\begin{proof}
Along the proof, we denote by $C$ a generic positive constant depending only on $\Om$, $d$, and $p$, but not on $\varepsilon$ and on $(u, m) \in \mathcal{A}_{\varepsilon}^{w}$.

Let us prove~\eqref{eq:rigidity-1}. For $(u, m) \in \mathcal{A}^w_\varepsilon$ we set $y_{\varepsilon} := id + \varepsilon u$. Applying the Geometric Rigidity Theorem~\ref{lem:gp-rigidity} and Lemma \ref{l:epsilon-well}, we deduce that there exists $R_{\varepsilon} \in SO(d)$ such that
\begin{equation*}
    \int_{\Omega} | \nabla y_{\varepsilon} - R_{\varepsilon} |^{2} \d x \leq C \int_{\Omega} g_p(\dist (\nabla y_{\varepsilon} ; SO(d))) \d x \leq C \varepsilon^{2} \big(\mathcal{G}_{\varepsilon} (u,m) + 1) .
\end{equation*}
By triangle inequality, we continue with
\begin{equation}
\label{e:new-2}
    \int_{\Omega} | \nabla y_{\varepsilon} - I |^{2} \d x \leq C \big(\varepsilon^{2} \big(\mathcal{G}_{\varepsilon} (u,m) + 1) +  | R_\varepsilon - I|^2 \big).
\end{equation}
Arguing as in~\cite[Proposition~3.4]{DalNegPer2002}, we get that
\begin{equation}
\label{e:new-3}
     | R_{\varepsilon} - I |^{2} \leq C \varepsilon^{2} \bigg(\mathcal{G}_{\varepsilon}(u,m) + \int_{ \Gamma } | w|^{2} \d \mathcal{H}^{d-1} +1\bigg),
\end{equation}
which, together with \eqref{e:new-2}, implies \eqref{eq:rigidity-1}.

As for \eqref{eq:rigidity-2}, we notice that, by Theorem \ref{lem:gp-rigidity} and Lemma~\ref{l:epsilon-well}, for the same rotation $R_\varepsilon$ it holds
\begin{equation*}
    \int_{\Omega} | \nabla y_{\varepsilon} - R_{\varepsilon} |^{p} \d x \leq C \int_{\Omega} g_p(\dist (\nabla y_{\varepsilon} ; SO(d))) \d x +C\varepsilon^2 \leq C \varepsilon^{2}\big(\mathcal{G}_{\varepsilon} (u,m) + 1\big).
\end{equation*}
Hence, recalling \eqref{e:new-3} we deduce that
\begin{align*}
    \int_{\Omega} | \varepsilon \nabla u|^{p}\d x & \leq 2^{p-1} \bigg(\int_{\Omega} | \nabla y_{\varepsilon} - R_{\varepsilon} |^{p} \d x +  | R_{\varepsilon} - I |^{p} \bigg) 
    \leq C \bigg(\varepsilon^{2} \mathcal{G}_{\varepsilon} (u,m) + \varepsilon^2 + | R_{\varepsilon} - I |^{p}  \bigg) 
    \\
    &
    \leq C \bigg(\varepsilon^{2} \mathcal{G}_{\varepsilon} (u,m) + \varepsilon^2+  | R_{\varepsilon} - I |^{2}  \bigg) \leq  
 C \varepsilon^{2} \bigg(\mathcal{G}_{\varepsilon} (u,m) +  \int_{ \Gamma }  | w|^2 \d \mathcal{H}^{d-1} +1  \bigg),\nonumber
\end{align*}
where we have used the fact that $p >d \geq 2 $ and that $R_{\varepsilon} \in SO(d)$, so that $| R_{\varepsilon} - I| \leq 2\sqrt{d}$. This concludes the proof of the proposition.
\end{proof}

We are now ready to prove the compactness result of Proposition~\ref{prop:compactness}.   We closely follow the arguments of~\cite[Proposition~4.3]{Bre2021} (see also \cite[Proposition 3.7]{Bresciani-Kruzik}). 

\begin{proof}[Proof of Proposition~\ref{prop:compactness}]
Let us notice that the Poincar\'e inequality and Lemma~\ref{prop:compactness_u} guarantee, up to a subsequence, weak convergence $u_{\varepsilon}\weakto u_{0}$ in $W^{1,2}(\Omega,\mathbb{R}^d)$.
Moreover, \eqref{eq:rigidity-2} 
implies that 
\begin{equation}
\label{e:y-id}
    \| \nabla y_{\varepsilon} - I\|_{L^{p}} \leq C \varepsilon^{\frac{2}{p}}
\end{equation}
for some positive constant~$C$ independent of~$\varepsilon$. 
And thus, strong convergence $y_{\varepsilon} \to id$ in $W^{1,p}(\Omega,\mathbb{R}^d)$.

For simplicity of notation, let us set $\Lambda_{\varepsilon}:= \exp(\varepsilon e(\nabla y_{\varepsilon}, m_{\varepsilon} \circ y_{\varepsilon})) \nabla y_{\varepsilon}$. Then, by the assumptions on~$\Phi$ we have that
\begin{align}
\label{e:equiintegrability}
    \int_{\Omega} \frac{1}{(\det \nabla y_{\varepsilon})^{a}} \d x &  = \int_{\Omega} \frac{1}{(\det \Lambda_{\varepsilon})^{a}} \d x 
    \\
    &
    \leq C \int_{\Omega} W(\nabla y_{\varepsilon}, m_{\varepsilon} \circ y_{\varepsilon}) \d x + C\mathcal{L}^{d}(\Omega) \leq C \varepsilon^{2} \mathcal{G}_{\varepsilon} (y_{\varepsilon}, m_{\varepsilon}) +  C\mathcal{L}^{d}(\Omega). \nonumber
\end{align}
Thus, $1/\det \nabla y_{\varepsilon}$ is equi-integrable in~$\Omega$.

For $\delta>0$ let us set
\[
    \Omega_{\delta}:= \{ x \in \Omega: \, \dist (x, \partial\Omega) >\delta\}
    \qquad \text{ and } \qquad
    \mu_{\varepsilon}^\delta :=  \chi_{\Omega_{\delta}} m_{\varepsilon}.
\]
Without loss of generality, we may assume that $\Omega_{\delta}$ is a Lipschitz set. Since $y_{\varepsilon} \to id$ in $W^{1, p} (\Omega; \bbR^{d})$ and $p >d$, we have that $y_{\varepsilon} \to id$ in $C^{0} (\overline\Omega; \bbR^{d})$. 
Hence, 
for $\varepsilon>0$ small enough, we may assume that $\Omega_{\delta} \subseteq y_{\varepsilon} (\Omega)$.
Without going into details about the degree theory, see \cite{FonGan1995}, let us provide rather a standard argument to prove this fact.
For $y_{\varepsilon}$, let us define the topological degree $\deg(x,y_\varepsilon,\Omega)$ as a topological degree of its continuous representative.
Notice that $h_{t}(x):=t y_{\varepsilon} (x) + (1-t) x$ is a homotopy between 
$y_{\varepsilon}$ and $id$, and  
by the uniform convergence $y_{\varepsilon}\to id$, for any $x\in \Omega_{\delta}$ it holds that $x\not\in y_\varepsilon(\partial \Omega)$ for any small enough $\varepsilon$.
Thus, $\deg(x,y_\varepsilon,\Omega) = \deg(x,id,\Omega) >0$, in other words, $x\in y_{\varepsilon} (\Omega)$.

Next, by change of variables, using the equality $|m_{\varepsilon} \circ y_{\varepsilon}| \det \nabla y_{\varepsilon} = 1$ a.e.~in~$\Omega$, we get that
\begin{align}
\label{e:mu}
    \int_{\bbR^{d}} |\mu_{\varepsilon}^\delta|^{2}   \d x  & = \int_{\Omega_{\delta}} | m_{\varepsilon}|^{2}  \d x 
    \\
    & \leq \int_{y_{\varepsilon}(\Omega)} | m_{\varepsilon}|^{2}  \d x
    = \int_{\Omega} | m_{\varepsilon} \circ y_{\varepsilon}|^{2} \, \det \nabla y_{\varepsilon}   \d x = \int_{\Omega} \frac{1}{\det \nabla y_{\varepsilon}}   \d x  .\nonumber
\end{align}
Moreover, by definition of~$\mu_{\varepsilon}^{\delta}$ it holds
\begin{align}
\label{e:mu-2}
    \int_{\Omega_{\delta}} | \nabla \mu_{\varepsilon}^\delta|^{2}   \d x \leq \int_{y_{\varepsilon}(\Omega)} | \nabla m_{\varepsilon}|^{2}  \d x \leq \mathcal{G}_{\varepsilon} (y_{\varepsilon}, m_{\varepsilon}) .
\end{align}
Therefore, $\mu_{\varepsilon}^\delta$ is bounded in $W^{1,2}(\Omega_{\delta}; \bbR^{d})$ uniformly for all $\delta>0$. 
By a diagonal argument, we find $m_0 \in W^{1,2}_{loc} (\Omega; \bbR^{d})$ such that $\mu_{\varepsilon}^{\delta} \rightharpoonup m_0$ weakly in $W^{1,2}_{loc} (\Omega; \bbR^{d})$.
While the Sobolev Embedding Theorems provide that 
\begin{equation}
\label{pointwise_conv_mu}
    m_{\varepsilon}(x) = \mu_{\varepsilon}^{\delta}(x) \to m_0(x)
    \qquad \text{for a.e.~}x\in\Omega_{\delta}. 
\end{equation}
Passing to the liminf as $\varepsilon \to 0$ and then to the limit as $\delta \to 0$ in~\eqref{e:mu} and~\eqref{e:mu-2}, we further deduce that $m_0 \in W^{1,2}(\Omega; \bbR^{d})$. 
Note further that the second line in \eqref{e:mu} ensures that 
$\chi_{y_{\varepsilon} (\Omega)} m_{\varepsilon}$ is bounded in $L^{2} (\bbR^{d}; \bbR^{d})$, and thus converges weakly to some $\mu \in L^{2} (\bbR^{d}; \bbR^{d})$. 
Using pointwise convergence~\eqref{pointwise_conv_mu}, we conclude 
$\chi_{y_{\varepsilon} (\Omega)} m_{\varepsilon}(x) \to \chi_\Omega m_0(x)$ a.e.~in $\bbR^d$.
Therefore $\mu = \chi_\Omega m_0$ a.e., and, up to a subsequence,
\begin{equation}
\label{weak_conv_m}
    \chi_{y_{\varepsilon} (\Omega)} m_{\varepsilon} \weakto \chi_{\Omega} m_0 \quad \text{ weakly in } L^{2}(\bbR^{d}; \bbR^{d}).
\end{equation}
Now it is not difficult to see that $|\chi_{y_{\varepsilon} (\Omega)} m_{\varepsilon}|^{2}$ converges in measure and, moreover, is equi-integrable by \eqref{e:equiintegrability} and \eqref{e:mu}.
Thus, 
applying the Lebesgue--Vitali Convergence Theorem (see \cite[Th.~4.5.4]{Bog2007} or \cite[Ch.~6, Ex.~10]{Rud1987}) to $|\chi_{y_{\varepsilon} (\Omega)} m_{\varepsilon}|^2$, we obtain convergence of norms
\[
    \|\chi_{y_{\varepsilon} (\Omega)} m_{\varepsilon}\|_{L^2} \to \|\chi_{\Omega} m_0\|_{L^2}.
\]
This, together with weak convergence~\eqref{weak_conv_m} guarantees strong convergence 
$\chi_{y_{\varepsilon} (\Omega)} m_{\varepsilon} \to \chi_{\Omega} m_{0}$ in $L^{2}(\bbR^{d}; \bbR^{d})$.
Since the weak convergence in~$L^{2}(\Om; \R^{d})$ of $m_{\varepsilon} \chi_{y_{\varepsilon} (\Om)}$ to~$m_0$ has been already shown, we deduce the strong $L^{2}$-convergence.
    
By~\eqref{e:energy-bound} we know that $ \chi_{y_{\varepsilon} (\Om)}\nabla m_{\varepsilon}$ is bounded in~$L^{2}(\R^{d}; \bbR^{d \times d})$. For every test $\varphi \in L^{2}(\R^{d}; \bbR^{d \times d})$ we consider
\begin{align*}
     \int_{\R^{d}} & (\chi_{y_{\varepsilon} (\Om)} \nabla m_{\varepsilon} -  \chi_{\Om} \nabla m_0) : \varphi \d x 
     \\
     &
     =  \int_{\Om_{\delta}} (\chi_{y_{\varepsilon} (\Om)} \nabla m_{\varepsilon}  - \chi_{\Om} \nabla m_0) : \varphi \d x +  \int_{\R^{d}\setminus \Om_{\delta}} (\chi_{y_{\varepsilon} (\Om)} \nabla m_{\varepsilon} -  \chi_{\Om} \nabla m_0) : \varphi \d x .
\end{align*}
Then, the first integral tends to~$0$ since $\Om_{\delta} \subseteq y_{\varepsilon} (\Om) \cap \Om$, $m_{\varepsilon} = \mu_{\varepsilon}^{\delta}$ in~$\Om_{\delta}$, and $\mu_{\varepsilon}^{\delta} \rightharpoonup m_0$ weakly in $W^{1,2}(\Om_{\delta}; \R^{d})$. As for the second integral, we have that
\[
    \bigg|\int_{\R^{d}\setminus \Om_{\delta}} (\chi_{y_{\varepsilon} (\Om)} \nabla m_{\varepsilon}  -  \chi_{\Om} \nabla m_0) : \varphi \d x\bigg| \leq C \| \varphi\|_{L^{2} (C_{\delta})} ,
\]
where $C_{\delta} = \{x \in \R^{d}: \, \dist (x, \partial\Om) < \delta\}$. Since $\mathcal{L}^{d} (C_{\delta}) \to 0$ as $\delta \to 0$, we infer that $\chi_{y_{\varepsilon} (\Om)} \nabla m_{\varepsilon} $ converges to $\chi_{\Om} \nabla m_0 $ weakly in $L^{2}(\R^{d}; \bbR^{d \times d})$.

We now show that $m_{\varepsilon} \circ y_{\varepsilon}$ converges to $m_{0}$ in $L^{1} (\Omega; \bbR^{d})$. For $\delta>0$ and $\varepsilon \ll1$ we write
\begin{align}
\label{e:triangle-m}
    \int_{\Omega_{\delta}} | m_{\varepsilon} \circ y_{\varepsilon} - m_0|   \d x \leq \int_{\Omega_{\delta}} | m_{\varepsilon} \circ y_{\varepsilon} - \mu_{\varepsilon}^{\delta}|   \d x + \int_{\Omega_{\delta}} | \mu_{\varepsilon}^{\delta} - m_0|   \d x . 
\end{align}
For the second integral on the right-hand side of~\eqref{e:triangle-m} we simply have that $\mu_{\varepsilon}^{\delta} \to m_0$ in~$L^{2} (\Omega_{\delta}; \bbR^{d})$. As for the first term on the right-hand side of~\eqref{e:triangle-m}, we notice that $m_{\varepsilon} \circ y_{\varepsilon}$ and $\mu_{\varepsilon}^{\delta}$ are equi-integrable in~$\Omega_{\delta}$. 
Thus, for $\eta >0$ there exists $\rho>0$ such that for every $A \subseteq \Omega_{\delta}$ measurable
\[
    \mathcal{L}^{d} (A) < \rho \ \Rightarrow \ \limsup_{\varepsilon \to 0} \int_{A} | m_{\varepsilon} \circ y_{\varepsilon} - \mu_{\varepsilon}^{\delta}|   \d x < \eta .
\]

Let us consider $\overline{\delta} \in (0, \delta)$ such that $\mathcal{L}^{d} (\Omega_{\overline\delta} \setminus \Omega_{\delta}) < \rho$ and set $A_{\delta, \varepsilon} := y^{-1}_{\varepsilon} (\Omega_{\delta})$. For $\varepsilon>0$ small enough, we have that $A_{\delta, \varepsilon} \subseteq \Omega_{\overline{\delta}}$. We rewrite
\begin{align}
\label{e:mu-3}
    \int_{\Omega_{\delta}} | m_{\varepsilon} \circ y_{\varepsilon} - \mu_{\varepsilon}^{\delta}|   \d x = \int_{\Omega_{\delta} \cap A_{\delta, \varepsilon}} | m_{\varepsilon} \circ y_{\varepsilon} - \mu_{\varepsilon}^{\delta}|   \d x + \int_{\Omega_{\delta} \setminus A_{\delta, \varepsilon}} | m_{\varepsilon} \circ y_{\varepsilon} - \mu_{\varepsilon}^{\delta}|   \d x .
\end{align}
We notice that, by change of variable,
\[
    \mathcal{L}^{d} (\Omega_{\delta}) = \mathcal{L}^{d} (y_{\varepsilon} (A_{\delta, \varepsilon})) = \mathcal{L}^{d} (A_{\delta, \varepsilon}) + \int_{A_{\delta, \varepsilon}} (\det \nabla y_{\varepsilon} - 1)   \d x .
\]
Since $\det \nabla y_{\varepsilon} \to 1$ in $L^{1}(\Omega)$, $A_{\delta, \varepsilon} \subseteq \Omega_{\overline{\delta}}$, and $\mathcal{L}^{d} (\Omega_{\overline\delta} \setminus \Omega_{\delta}) < \rho$, we conclude that
\begin{align*}
    \limsup_{\varepsilon \to 0} \mathcal{L}^{d} (\Omega_{\delta}\setminus A_{\delta, \varepsilon}) & \leq \limsup_{\varepsilon \to 0} \Big(\mathcal{L}^{d} (\Omega_{\overline{\delta}}) - \mathcal{L}^{d} (A_{\delta, \varepsilon})\Big)
    \\
    &
    = \limsup_{\varepsilon \to 0} \Bigg(\mathcal{L}^{d} (\Omega_{\overline{\delta}}\setminus \Omega_{\delta}) - \int_{A_{\delta, \varepsilon}} (\det \nabla y_{\varepsilon} - 1)   \d x \Bigg)
    < \rho .
\end{align*}
In turn, this implies that 
\begin{align}
\label{e:mu-4}
    \limsup_{\varepsilon \to 0} \int_{\Omega_{\delta} \setminus A_{\delta, \varepsilon}} | m_{\varepsilon} \circ y_{\varepsilon} - \mu_{\varepsilon}^{\delta}|   \d x < \eta .
\end{align}
To estimate the integral on~$\Om_{\delta} \cap A_{\delta,\varepsilon}$, we fix $\lambda \in (-1/p, 0)$, extend~$\mu_{\varepsilon}^{\delta}$ from~$\Om_{\delta}$ to a map~$M_{\varepsilon}^{\delta} \in W^{1, 2}(\R^{d}; \R^{d})$ with~$\| M_{\varepsilon}^{\delta}\|_{W^{1, 2}(\R^{d})} \leq C(\Omega_{\delta}) \| \mu_{\varepsilon}^{\delta}\|^{2}_{W^{1, 2} (\Om_{\delta})}$, and use the Lusin approximation of Sobolev functions: for every $\varepsilon>0$ there exists a set $G_{\delta,\varepsilon}$ such that $M_{\varepsilon}^{\delta}$ is $\varepsilon^{\lambda}$-Lipschitz on~$G_{\delta,\varepsilon}$ and
\begin{align}
\label{e:mu-5}
    \mathcal{L}^{d} (\R^{d} \setminus G_{\delta,\varepsilon}) \leq \varepsilon^{-2\lambda} \int_{\R^{d}} | \nabla M_{\varepsilon}^{\delta}|^{2}  \d x \leq C(\Omega_{\delta}) \varepsilon^{-2\lambda} \| \mu_{\varepsilon}^{\delta}\|^{2}_{W^{1, 2} (\Om_{\delta})}  .
\end{align}
Setting~$X_{\delta,\varepsilon}:= y_{\varepsilon}^{-1} (G_{\delta,\varepsilon})$ and noticing that $m_{\varepsilon} = M_{\varepsilon}^{\delta}$ on $\Omega_{\delta} \cap G_{\delta,\varepsilon}$, we have that
\begin{align*}
    &\int_{\Om_{\delta} \cap A_{\delta, \varepsilon} \cap X_{\delta,\varepsilon} \cap G_{\delta,\varepsilon}}  | m_{\varepsilon} \circ y_{\varepsilon} - \mu_{\varepsilon}^{\delta}|   \d x  \\
    &
    \nonumber\qquad = \int_{\Om_{\delta} \cap A_{\delta, \varepsilon} \cap X_{\delta,\varepsilon} \cap G_{\delta,\varepsilon}} | M_{\varepsilon}^{\delta} \circ y_{\varepsilon} - M_{\varepsilon}^{\delta}|   \d x
     \leq \varepsilon^{\lambda} \mathcal{L}^{d} (\Om) \|  y_{\varepsilon} - id\|_{C^{0}} \leq  \varepsilon^{\frac{2}{p} + \lambda} \mathcal{L}^{d} (\Om) .
\end{align*}
Combining~\eqref{e:mu-5}, the convergence rate~\eqref{e:y-id}, and a change of variable, we infer that
\[
    \lim_{\varepsilon \to 0} \mathcal{L}^{d} (\Om_{\delta} \setminus G_{\delta,\varepsilon}) = 0 
    \qquad \text{and}\qquad 
    \lim_{\varepsilon \to 0} \mathcal{L}^{d} (\Om_{\delta} \setminus X_{\delta,\varepsilon}) = 0 . 
\]
Thus, 
\begin{align}
\label{e:mu-7}
    \limsup_{\varepsilon \to 0} \int_{\Om_{\delta} \cap A_{\delta,\varepsilon} \setminus (X_{\delta,\varepsilon} \cap G_{\delta,\varepsilon}}) | m_{\varepsilon}^\delta \circ y_{\varepsilon} - \mu_{\varepsilon}^\delta| \d x < \eta .
\end{align}

Combining~\eqref{e:triangle-m}, \eqref{e:mu-3}, \eqref{e:mu-4}, and~\eqref{e:mu-7} yields
\[
    \limsup_{\varepsilon \to 0} \int_{\Om_{\delta}} | m_{\varepsilon} \circ y_{\varepsilon} - m_0|  \d x < 3\eta .
\]
Since $\eta>0$ is arbitrary, this implies that $m_{\varepsilon} \circ y_{\varepsilon} \to m_0$ in $L^{1} (\Om_{\delta}; \R^{d})$ for every $\delta>0$. By the equi-integrability of $m_{\varepsilon} \circ y_{\varepsilon}$, this entails the convergence in~$L^{1}(\Om; \R^{d})$. 
Furthermore,
since strong $L^1$-convergence implies convergence a.e., we have
$m_{\varepsilon} \circ y_{\varepsilon} \det \nabla y_{\varepsilon} \to m_0$ a.e. and $|m_{\varepsilon} \circ y_{\varepsilon}| \det \nabla y_{\varepsilon} = |m_0|=1$ a.e.~in $\Omega$.
Thus, 
$m \in W^{1, 2} (\Om; \bbS^{d-1})$
and $m_{\varepsilon} \circ y_{\varepsilon} \det \nabla y_{\varepsilon} \to m_0$ in $L^r(\Omega;\bbR^d)$ for any $r \in [ 1, +\infty)$.
\end{proof}


\section{Proof of $\Gamma$-convergence}
\label{s:conv}

We are now able to prove Theorem~\ref{thm:gamma-convergence}.

\begin{proof}[Proof of Theorem \ref{thm:gamma-convergence}]
We divide the proof in lower and upper bound for the $\Gamma$-convergence.

\noindent{\em $\Gamma$-liminf inequality.} We start with the lower bound. Let $(u_{\varepsilon}, m_{\varepsilon}) \in \mathcal{A}_{\varepsilon}^{w}$ be such that
\[
    \sup_{\varepsilon>0} \, \mathcal{G}_{\varepsilon} (u_{\varepsilon}, m_{\varepsilon}) < +\infty .
\]
By Proposition~\ref{prop:compactness}, there exist 
$u \in W^{1, 2}_w(\Om; \R^{d})$ and $m \in W^{1, 2}(\Om; \mathbb{S}^{d-1})$ 
such that, up to a subsequence, 
$u_{\varepsilon} \rightharpoonup u$ weakly in $W^{1, 2} (\Om; \R^{d})$, 
$m_{\varepsilon} \circ y_{\varepsilon} \, \det \nabla y_{\varepsilon} \to m$ in $L^{r}(\Om; \R^{d})$ for any $1\leq r <\infty$, 
$m_{\varepsilon} \circ y_{\varepsilon} \to m$ in $L^{1}(\Om; \R^{d})$, $\chi_{y_{\varepsilon}(\Omega)} \nabla m_{\varepsilon}  \rightharpoonup \chi_{\Om} \nabla m $ weakly in~$L^{2}(\R^{d}; \bbR^{d \times d })$, 
and $\chi_{y_{\varepsilon}(\Om)} m_{\varepsilon}  \to \chi_{\Om} m $ in $L^{2}(\R^{d}; \R^{d})$. 
Then, by Proposition~\ref{prop:magn} we have that
\begin{equation*}
    \int_{\Om} |\nabla m|^{2}  \d x + \int_{\R^{d}} | \nabla v_{m}|^{2}  \d x \leq \liminf_{\varepsilon \to 0}  \int_{y_{\varepsilon} (\Om) } |\nabla m_{\varepsilon}|^{2}  \d z + \int_{\R^{d}} | \nabla v_{m_{\varepsilon}}|^{2}  \d z .
\end{equation*}

For a fixed $0<\alpha<1$ define the sets
\[
    E_{\varepsilon}:=\{x\in\Omega:|\nabla u_{\varepsilon} (x)| < \varepsilon^{-\alpha}\}.
\]
Then, by the Chebyshev inequality
\[
    \mathcal{L}^{d} (\Omega \setminus E_{\varepsilon}) = 
    \mathcal{L}^{d} \big(\left\{x\in\Omega:|\nabla u_{\varepsilon} (x)| \geq \varepsilon^{-\alpha}\right\} \big) \leq
    \frac{1}{\varepsilon^{-\alpha}} \int_{\Omega \setminus E_{\varepsilon}} |\nabla u_{\varepsilon} (x)|   {\rm d} x \leq 
    \varepsilon^{\alpha} |\Omega|^{1/2} \|\nabla u_{\varepsilon}\|_{L^2}
    \leq C \varepsilon^{\alpha}.
\]
Thus, $\chi_{E_\varepsilon} \to 1$ in measure
and $\chi_{E_{\varepsilon}} \nabla u_{\varepsilon} \weakto \nabla u$ in $L^{2}(\Omega; \bbR^{d \times d})$. 
Indeed, 
for $1\leq r < 2$ we may estimate 
\begin{align*} 
    \int_{\Omega\setminus E_\varepsilon} |\nabla u_{\varepsilon} - \nabla u|^r \d x & 
    \leq
    \|\nabla u_{\varepsilon} - \nabla u\|^{\frac{1}{r}}_{L^2} \cdot \mathcal{L}^{d} (\Omega \setminus E_\varepsilon)^{\frac{2-r}{2r}} 
    \leq C \varepsilon^{\frac{(2-r)\alpha}{2r}} \to 0.
\end{align*}
Therefore,
$\chi_{E_\varepsilon}\nabla u_{\varepsilon} = \nabla u_{\varepsilon}  - \chi_{\Omega \setminus E_\varepsilon}\nabla u_{\varepsilon}  \weakto \nabla u \text{ weakly in } L^r (\Om;  \bbR^{d \times d})$.
Since the weak limit is unique, we get that $\chi_{E_{\varepsilon}} \nabla u_{\varepsilon} \rightharpoonup \nabla u$ weakly in $L^{2}(\Omega; \bbR^{d \times d})$. 

Moreover, since  $m_\varepsilon \circ {y}_{\varepsilon} \det \nabla y_{\varepsilon} \to m $ strongly in $L^r(\Omega;\bbR^d)$  for every $1\leq r < \infty$
then 
$(\det \nabla y_{\varepsilon})^2 m_{\varepsilon}\circ y_{\varepsilon}\otimes m_{\varepsilon}\circ y_{\varepsilon} \to m\otimes m$  strongly in  $L^r(\Omega;\bbR^d)$. Thus, we have that $e (\nabla y_{\varepsilon}, {m_{\varepsilon}\circ y_{\varepsilon}}) \to e (m)$ and  $\chi_{E_{\varepsilon}}  e (\nabla y_{\varepsilon}, {m_{\varepsilon}\circ y_{\varepsilon}}) \to e (m)$ strongly in  $L^2(\Omega;\bbR^d)$.

Let us set 
\[
    \beta(\varepsilon):= (\exp(\varepsilon e(\nabla y_{\varepsilon}, m_{\varepsilon} \circ y_{\varepsilon}) - I - \varepsilon e(\nabla y_{\varepsilon}, m_{\varepsilon} \circ y_{\varepsilon})) \nabla y_{\varepsilon} + \varepsilon^{2} e(\nabla y_{\varepsilon}, m_{\varepsilon} \circ y_{\varepsilon}) \nabla u_{\varepsilon}.
\]
Recalling that $\| e (\nabla y_{\varepsilon}, m_{\varepsilon}\circ y_{\varepsilon}) \|_{\infty}$ and $\| \exp (\varepsilon e (\nabla y_{\varepsilon}, m_{\varepsilon}\circ y_{\varepsilon})\|_{\infty}$ are uniformly bounded
and $\Phi(I) = D \Phi(I) = 0$, by Taylor expansion we have that
\begin{align}
\label{e:liminf-many}
    \frac{1}{\varepsilon^{2}} & \int_{\Omega} W(I +  \varepsilon \nabla u_{\varepsilon}, m_{\varepsilon} \circ y_{\varepsilon})\d x 
    \\ 
    & = \frac{1}{\varepsilon^{2}} \int_{E_{\varepsilon}} W(I + \varepsilon \nabla u_{\varepsilon}, m_{\varepsilon} \circ y_{\varepsilon})\d x + 
    \frac{1}{\varepsilon^{2}} \int_{\Omega \setminus E_{\varepsilon}} W(I + \varepsilon \nabla u_{\varepsilon}, m_{\varepsilon} \circ y_{\varepsilon})\d x 
    \nonumber
    \\
    & 
    \geq \frac{1}{\varepsilon^{2}} \int_{E_{\varepsilon}} W(I + \varepsilon \nabla u_{\varepsilon}, m_{\varepsilon} \circ y_{\varepsilon})\d x =  
    \frac{1}{\varepsilon^{2}} \int_{E_{\varepsilon}} \Phi (\exp(\varepsilon e (\nabla y_{\varepsilon}, m_{\varepsilon}\circ y_{\varepsilon}) (I + \varepsilon \nabla u_{\varepsilon}))\d x 
    \nonumber
    \\
    & 
    = \frac{1}{\varepsilon^{2}} \int_{E_{\varepsilon}}  \Phi \big(I + \varepsilon (\nabla u_{\varepsilon} + e (\nabla y_{\varepsilon}, m_{\varepsilon}\circ y_{\varepsilon})) + \beta(\varepsilon) \big) \d x \nonumber
    \\
    & 
    \geq \frac{1}{{2}\varepsilon^{2}} \int_{E_{\varepsilon}}  \bbC (\varepsilon(\strain(u_{\varepsilon}) + e (\nabla y_{\varepsilon} , m_{\varepsilon}\circ y_{\varepsilon})) + \beta(\varepsilon)) : (\varepsilon(\strain (u_{\varepsilon}) + e (\nabla y_{\varepsilon}, m_{\varepsilon}\circ y_{\varepsilon})) + \beta(\varepsilon)) \d x \nonumber
    \\
    & \hspace{2em}
    - \int_{E_{\varepsilon}} \frac{\eta(| \varepsilon (\nabla u_{\varepsilon} + e (\nabla y_{\varepsilon}, m_{\varepsilon}\circ y_{\varepsilon})) + \beta(\varepsilon)|) | \varepsilon (\nabla u_{\varepsilon} + e (\nabla y_{\varepsilon}, m_{\varepsilon}\circ y_{\varepsilon}))  + \beta(\varepsilon)|^2}{\varepsilon^{2}}\d x \nonumber
    \\
    & 
    = \frac{1}{{2}} \int_{E_{\varepsilon}}  \bbC (\strain (u_{\varepsilon}) + e (\nabla y_{\varepsilon}, m_{\varepsilon}\circ y_{\varepsilon})) : (\strain (u_{\varepsilon})+ e (\nabla y_{\varepsilon}, m_{\varepsilon}\circ y_{\varepsilon})) \d x \nonumber
    \\
    & \hspace{2em}
    - 2\int_{E_{\varepsilon}} \eta(| \varepsilon (\nabla u_{\varepsilon} + e (\nabla y_{\varepsilon}, m_{\varepsilon}\circ y_{\varepsilon}) + \beta(\varepsilon))|) (|\nabla u_{\varepsilon} + e (\nabla y_{\varepsilon}, m_{\varepsilon}\circ y_{\varepsilon})|^{2} + \varepsilon^{-2} |\beta(\varepsilon)|^2)\d x \nonumber
    \\
    &
    \hspace{2em}
    - \frac{1}{2\varepsilon^{2}} \int_{E_{\varepsilon}} \mathbb{C} \beta(\varepsilon) : \beta(\varepsilon)\, \d x + \frac{1}{\varepsilon}\int_{E_{\varepsilon}} \mathbb{C} (\strain (u_{\varepsilon}) + e(\nabla y_{\varepsilon}, m_{\varepsilon}\circ y_{\varepsilon})) : \beta(\varepsilon) \d x. \nonumber
\end{align}
Since $| \nabla u_{\varepsilon}| \leq \varepsilon^{-\alpha}$ in $E_{\varepsilon}$ and $\alpha \in (0, 1)$, we have that $\| \varepsilon \nabla u_{\varepsilon} + e(\nabla y_{\varepsilon}, m_{\varepsilon} \circ y_{\varepsilon}) \|_{L^{\infty} (E_{\varepsilon})}$ is bounded. Hence, we estimate
\begin{equation}
\label{e:liminf-many2}
    \| \beta(\varepsilon)\|_{L^{\infty} (E_{\varepsilon})} \leq C \varepsilon^{2} (1 + \varepsilon^{-\alpha}). 
\end{equation}
for some $C>0$ independent of $\varepsilon$.
In particular, also $\| \beta(\varepsilon)\|_{L^{\infty} (E_{\varepsilon})}$ is bounded. Thus, up to a redefinition of~$C$, we have that 
\begin{align}
\label{e:liminf-many3}
    \| \eta&  (| \varepsilon (\nabla u_{\varepsilon} + e (\nabla y_{\varepsilon}, m_{\varepsilon}\circ y_{\varepsilon}) + \beta(\varepsilon))|) \|_{L^{\infty} (E_{\varepsilon})} 
    \\
    &
    \leq C \| \varepsilon (\nabla u_{\varepsilon} + e (\nabla y_{\varepsilon}, m_{\varepsilon}\circ y_{\varepsilon}) + \beta(\varepsilon))\|_{L^{\infty} (E_{\varepsilon})}
    \leq C \varepsilon^{1-\alpha} \nonumber
\end{align}
for $\varepsilon \in (0, 1)$.
Combining \eqref{e:liminf-many2} and \eqref{e:liminf-many3} and recalling that $e(\nabla y_{\varepsilon}, m_{\varepsilon} \circ y_{\varepsilon}) \to e(m)$ in $L^{2} (\Omega; \R^{d})$ (see Proposition~\ref{prop:compactness}), passing to the liminf in~\eqref{e:liminf-many} we infer that
\begin{align*}
    \liminf_{\varepsilon \to 0}\frac{1}{\varepsilon^{2}} \int_{\Omega} W(I & + \varepsilon \nabla u_{\varepsilon}, m_{\varepsilon} \circ y_{\varepsilon})\d x 
    \geq \frac{1}{{2}} \int_{\Omega} \bbC (\strain (u) + e (m)) : (\strain (u) + e(m))  \d x. 
\end{align*}

\noindent{\em $\Gamma$-$\limsup$ inequality.} We now construct a recovery sequence. 
Let $u_0 \in W^{2,\infty}_w(\Omega;\bbR^d)$ and $m_0 \in W^{1,2}(\Omega;\bbS^{d-1})$.
Define 
\begin{align}
    & y_{\varepsilon}(x)  = x + \varepsilon u_0(x), \label{def:y} \\
    & m_{\varepsilon}(x + \varepsilon u_0(x)) \det \nabla y_{\varepsilon}(x)  = m_0 (x). \label{def:m}
\end{align}
We want to prove that 
\begin{align}
    & y_\varepsilon \to id \text{ in } W^{1,p}(\Omega;\bbR^d),\label{conv:y}\\
    & u_\varepsilon = u_0 \to u_0 \text{ in } W^{1,2}(\Omega;\bbR^d),\label{conv:u}\\
    & \chi_{y_\varepsilon(\Omega)} m_\varepsilon \to \chi_\Omega m_0 \text{ in } L^{2}(\bbR^d;\bbR^d), \label{conv:m}\\ 
    & \chi_{y_\varepsilon(\Omega)} \nabla m_\varepsilon \to \chi_\Omega \nabla m_0 \text{ in }L^{2}(\bbR^d;\bbR^{d \times d}), \label{conv:nabla_m}\\ 
    & \limsup_{\varepsilon\to 0} \mathcal{G}_{\varepsilon}(u_{\varepsilon},m_{\varepsilon}) \leq \mathcal{G}_0(u_{0},m_{0}).
\end{align}

First, we notice that $y_{\varepsilon} \to id$ in $W^{2, \infty} (\Omega; \mathbb{R}^{d})$, since $\|y_\varepsilon - id\|_{W^{2, \infty}} \leq \varepsilon\|u_0\|_{W^{2, \infty}}$. This implies the convergences \eqref{conv:y} and \eqref{conv:u}.


The injectivity of $y_{\varepsilon}$ and the sign condition $\det \nabla y_{\varepsilon} >0$ follow from Theorem~\ref{thm:injectivity} ({\cite[Theorem 5.5-1]{Cia1988}}), since $\varepsilon \|\nabla u_0\|_{\infty} \leq c(\Omega)$ for small enough $\varepsilon$.
Moreover, 
$y_{\varepsilon}$ is bi-Lipschitz for small enough $\varepsilon$.
Indeed, there exist constants $0 < l < L < +\infty$ (that may depend on $u_0$ but not on $x$ and $\varepsilon$), such that for $\varepsilon>0$ small enough it holds that
$\sup_{x\in \Omega} |\nabla y_{\varepsilon}| = \sup_{x\in \Omega} |I + \varepsilon \nabla u_0(x)| \leq 1+ \varepsilon \sup_{x\in \Omega} |\nabla u_0(x)|  =: L_\varepsilon < L $ with $L_\varepsilon \to 1$ if $\varepsilon \to 0$.
At the same time, $\inf_{x\in \Omega} \det \nabla y_{\varepsilon} = \inf_{x\in \Omega}\det(I + \varepsilon \nabla u_0) = 1 + \varepsilon  \inf_{x\in \Omega} \tr \nabla u_0(x) + o(\varepsilon) = : l_\varepsilon > l>0$ and $l_\varepsilon \to 1$ if $\varepsilon \to 0$.
By the inverse function theorem $y_{\varepsilon}^{-1}$ is differentiable and 
$\nabla y_{\varepsilon}^{-1}(y_{\varepsilon}(x)) = (\nabla y_{\varepsilon}(x))^{-1} \approx I - \varepsilon \nabla u_0(x)$.
Thus, $m_\varepsilon$ defined in \eqref{def:m} can be explicitly written as 
\[
    m_{\varepsilon}(z) = \frac{1}{\det \nabla y_\varepsilon} m_0 \circ y_{\varepsilon}^{-1}(z) \text{ for } z\in y_\varepsilon(\Omega).
\]
Hence, we have that 
\[
    |m_{\varepsilon}(z)| = \left|\frac{m_0 (y_{\varepsilon}^{-1}(z))}{\det \nabla y_\varepsilon(y_{\varepsilon}^{-1}(z))}\right| \leq \frac{1}{l_\varepsilon}.
\]
By direct computation, we have that 
\begin{align*}
    \nabla \left(\frac{m_{0}}{\det \nabla y_\varepsilon}\right) &
    = \frac{1}{\det \nabla y_\varepsilon} \nabla m_{0} + m_{0} \otimes  \nabla \left(\frac{1}{\det \nabla y_\varepsilon} \right) \\
    & = \frac{1}{\det \nabla y_\varepsilon} \nabla m_{0} - \frac{1}{(\det \nabla y_\varepsilon)^2} m_{0} \otimes \nabla (\det \nabla y_\varepsilon).
\end{align*}
This implies that 
\begin{align*}
    \bigg \| \nabla \left(\frac{m_{0}}{\det \nabla y_\varepsilon}\right)\bigg\|_{L^{2}} & \leq  
    \frac{\| \nabla m_{0}\|_{L^{2}}}{l_\varepsilon} + C \frac{\| \nabla \det \nabla y_{\varepsilon}\|_{L^{\infty}}}{l_{\varepsilon}}, 
\end{align*}
for some positive constant $C$ only depending on~$\Omega$. 
This inequality, together with the fact that 
$ \left|\nabla y^{-1}_\varepsilon(z)\right| \leq \frac{L_\varepsilon^{d-1}}{l_\varepsilon} $,
provides 
\begin{align}
\label{e:recovery}
    \int_{y_{\varepsilon}(\Omega)}|\nabla m_{\varepsilon}(z)|^2 \d z
    & = 
    \int_{y_{\varepsilon}(\Omega)}\left|\nabla \left(\frac{m_{0}(y^{-1}_\varepsilon(z))}{\det \nabla y_\varepsilon(y^{-1}_\varepsilon(z))}\right) \cdot \nabla y^{-1}_\varepsilon(z)\right|^2 \det{\nabla y_\varepsilon}(y^{-1}_\varepsilon(z)) \d z \\
    & \leq
    \int_{\Omega}\left|\nabla \left(\frac{m_{0}(x)}{\det \nabla y_\varepsilon(x)}\right)\right|^2  \left|\nabla y^{-1}_\varepsilon(y_\varepsilon(x))\right|^2 \det \nabla y_\varepsilon(x) \d x \nonumber\\
    &\leq 
    \frac{2 L_\varepsilon^{3d-2}}{l_\varepsilon^4} \|\nabla m_0 (x) \|^2_{L^{2}}  +  C \frac{\| \nabla \det \nabla y_{\varepsilon}\|^{2}_{L^{\infty}}}{l_{\varepsilon}^{2}}. \nonumber
\end{align}
Thus, $m_\varepsilon \in W^{1,2}(y_\varepsilon(\Omega); \bbR^{d})$.  

Let us define $\Omega_{\delta}:= \{ x \in \Omega: \, \dist (x, \partial\Omega) >\delta\}$. In particular, $\Omega_{\delta} \subset \Omega \cap y_{\varepsilon} (\Omega)$ for $\varepsilon>0$ small enough, due to the uniform convergence of $y_{\varepsilon}$ to the identity. For such $\varepsilon$ and for $z\in\Omega_\delta$ we estimate
\begin{align*}
    |m_\varepsilon(z) - m_0(z)| & \leq \left| \frac{m_0 (y_{\varepsilon}^{-1}(z))}{\det \nabla y_\varepsilon(y_{\varepsilon}^{-1}(z))} - \frac{m_0 (z)}{\det \nabla y_\varepsilon(y_{\varepsilon}^{-1}(z))}\right| + 
    \left| \frac{m_0 (z)}{\det \nabla y_\varepsilon(y_{\varepsilon}^{-1}(z))}- m_0(z) \right| \\
    & \leq 
    \frac{1}{\det \nabla y_\varepsilon(y_{\varepsilon}^{-1}(z))}\left| m_0 (y_{\varepsilon}^{-1}(z)) - m_0 (z)\right| + 
    |m_0 (z)|\cdot\left| \frac{1}{\det \nabla y_\varepsilon(y_{\varepsilon}^{-1}(z))}- 1 \right|.
\end{align*}
Therefore, taking the squares, integrating over $\Omega_{\delta}$, and letting $\varepsilon \to 0$ we infer that $m_{\varepsilon} \to m_{0}$ in $L^{2} (\Omega_{\delta}; \R^{d})$, for every $\delta>0$. 
Arguing as in the proof of Proposition~\ref{prop:compactness}, it is not difficult to see that 
$\chi_{y_\varepsilon(\Omega)} m_\varepsilon \to \chi_\Omega m_0$ strongly in $L^{2}(\bbR^d;\bbR^d)$
and $\chi_{y_\varepsilon(\Omega)} \nabla m_\varepsilon \to \chi_\Omega \nabla m_0$ weakly in $L^{2}(\bbR^d;  \bbR^{d \times d})$.
Moreover, by~\eqref{e:recovery} we can estimate the exchange energy as 
\begin{align*}
    \limsup_{\varepsilon \to 0} \int_{y_\varepsilon(\Omega)} |\nabla m_{\varepsilon}(y)|^2 \d z 
    & \leq 
    \limsup_{\varepsilon \to 0} \int_{\Omega} \frac{L_\varepsilon^{3d-2}}{l_\varepsilon^4} |\nabla m_0 (x) |^2 \d x + C \frac{\| \nabla \det \nabla y_{\varepsilon}\|^{2}_{L^{\infty}}}{l_{\varepsilon}^{2}} = \int_{\Omega} |\nabla m_0 (x) |^2 \d x.
\end{align*}

For a.e.\ $x\in\Omega$, the elastic energy has the following form by Taylor's expansion and by definition of $m_{\varepsilon}$
\begin{align*}
    \frac{1}{\varepsilon^{2}}  W(I + \varepsilon \nabla u_0, m_\varepsilon \circ y_{\varepsilon})
    & =
    \frac{1}{\varepsilon^{2}} \Phi (\exp(\varepsilon e (\nabla y_{\varepsilon}, m_\varepsilon \circ y_\varepsilon)) (I + \varepsilon \nabla u_0)) 
    \\
    &
    =
    \frac{1}{\varepsilon^{2}} \Phi (\exp (\varepsilon e (m_0)) (I + \varepsilon \nabla u_0))
    \\
    & =
    \frac{1}{\varepsilon^{2}} \Phi(I + \varepsilon (\nabla u_0 + e (m_0)) + o(\varepsilon)) 
    \\
    &
    = \frac{1}{{2}} \bbC(\strain (u_0) + e (m_0)): (\strain (u_0) + e (m_0)) + O(1)
\end{align*}
with $O(1) \to 0$ as $\varepsilon \to 0$ uniformly in $\Om$.
Then \hyperref[hyp:d]{(d)} and the Dominated Convergence Theorem give
\begin{align*}
    \limsup\limits_{\varepsilon \to 0}\frac{1}{\varepsilon^{2}} \int_\Omega W(I + \varepsilon \nabla u_0, m_\varepsilon \circ y_{\varepsilon}) \d x
    & = \frac{1}{{2}} \int_\Omega \bbC (\strain (u_0) + e (m_{0})) : (\strain (u_0) +  e (m_0))  \d x. 
\end{align*}

We conclude that
\begin{align*}
    \limsup_{\varepsilon \to 0} \mathcal{G}_\varepsilon(u_\varepsilon, m_\varepsilon) 
    & = \limsup_{\varepsilon \to 0} \frac{1}{\varepsilon^{2}} \int_{\Omega} W(I + \varepsilon \nabla u_0, m_0) \d x + \frac{1}{2}\int_{y_\varepsilon(\Omega)} | \nabla m_\varepsilon|^{2} \d z + \frac{\mu_0}{2} \int_{\bbR^d} |\nabla v_{m_\varepsilon}|^2 \d z \\
    & \leq
    \frac{1}{{2}} \int_{\Omega} \bbC (\strain (u_0)  + e(m_0)) : (\strain (u_0) + e (m_0))  \d x + \frac{1}{2}\int_{\Omega} | \nabla m_0|^{2} \d x + \frac{\mu_0}{2} \int_{\bbR^d} |\nabla v_{m_0}|^2 \d z \\
    & \vphantom{\int_{\Omega}}= \mathcal{G}_{0}(u_0, m_0),
\end{align*}
where, for the convergence of $\nabla v_{m_{\varepsilon}}$, we have used Proposition~\ref{prop:magn}.

For a general $u_0 \in W^{1,2}_w(\Omega;\bbR^d)$ we conclude by a standard diagonal argument.
\end{proof}

We conclude this section with the proof of convergence of minimizers of $\mathcal{F}_{\varepsilon}$ to minimizers of $\mathcal{F}_{0}$, namely,  Theorems~\ref{thm:linearization} and~\ref{thm:convergence_of_minimizers}. We start with the coercivity of~$\mathcal{F}_{\varepsilon}$.

\begin{proposition}\label{prop:f_to_g}
    Let $(u_{\varepsilon}, m_{\varepsilon})\in \mathcal{A}^{w}_{\varepsilon}$ and $K >0$ be such that  
    \[
        \sup_{\varepsilon>0} \mathcal{F}_\varepsilon (u_\varepsilon,m_\varepsilon) \leq K.
    \]
    Then, there exists $C(K) >0$ such that
    \[
        \sup_{\varepsilon>0} \mathcal{G}_\varepsilon (u_\varepsilon,m_\varepsilon) \leq C(K).
    \]
\end{proposition}
\begin{proof}  
    Recall that 
    $\mathcal{F}_\varepsilon (u_\varepsilon,m_\varepsilon):=\mathcal{G}_\varepsilon (u_\varepsilon,m_\varepsilon) - \mathscr{L}(u_\varepsilon) - \mathscr{M}(y_\varepsilon,m_\varepsilon)$,
    with $|\mathscr{L}(u_\varepsilon)| \leq C_L \|u_\varepsilon\|_{L^2(\Omega;\bbR^d)}$
    and 
    $|\mathscr{M}(y_\varepsilon,m_\varepsilon)| \leq C_M\|m_\varepsilon\|_{L^2(y_\varepsilon(\Omega);\bbR^d)}$.
    Moreover, by Lemma~\ref{prop:compactness_u} we have
    \[
        \|u_\varepsilon\|_{L^2(\Omega;\bbR^d)}^2 \leq C \bigg(\mathcal{G}_\varepsilon (u_\varepsilon,m_\varepsilon) +  \int_{\Gamma} |w|^{2} \, \d \mathcal{H}^{d-1} + 1 \bigg),
    \] 
    while arguing as in~\eqref{e:mu} we deduce from the assumption \hyperref[hyp:e]{(e)} that
    \[
        \|m_\varepsilon\|_{L^2(y_\varepsilon(\Omega);\bbR^d)}^2 =
        \int_\Omega \frac{1}{\det \nabla y_\varepsilon} \d x
        \leq C(\mathcal{G}_\varepsilon (u_\varepsilon,m_\varepsilon) + 1).
    \]
    Summing up the estimates above, we obtain
    \[
        \mathcal{G}_\varepsilon (u_\varepsilon,m_\varepsilon)=\mathcal{F}_\varepsilon (u_\varepsilon,m_\varepsilon) + \mathscr{L}(u_\varepsilon) + \mathscr{M}(y_\varepsilon,m_\varepsilon)
        \leq K + C\sqrt{\mathcal{G}_\varepsilon (u_\varepsilon,m_\varepsilon)} + C,
    \]
    This concludes the proof of the proposition.
\end{proof}

We are now able to prove Theorem~\ref{thm:linearization}.

\begin{proof}[Proof of Theorem~\ref{thm:linearization}]
Consider a sequence 
$\varepsilon_k \to 0$ and recall that 
\[
    s_{\varepsilon}:=\inf\{\mathcal{F}_\varepsilon (u,m) : (u,m) \in \mathcal{A}^w_\varepsilon\}
    \qquad\text{ and }\qquad
    s_{0}:=\inf\{\mathcal{F}_0 (u,m) : (u,m) \in \mathcal{A}^w_0\}.    
\]
It is standard to show that  $\mathcal{F}_0$ has a minimizer $(u_0, m_0)\in W^{1,2}_w(\Omega;\bbR^d) \times W^{1,2}(\Omega;\bbS^{d-1})$ on $\mathcal{A}_0^w$.
It is also straightforward to check that 
$\inf \mathcal{F}_{\varepsilon_k} (u,m)$ is bounded with respect to $\varepsilon$.
Since the functionals $\mathscr{L}$ and $\mathscr{M}$ do not depend on $\varepsilon$,
from $\Gamma$-convergence of $\mathcal{G}_\varepsilon$
we deduce that  
$s_{\varepsilon_k} \to s_0$
and thus $\mathcal{F}_{\varepsilon_k} (u_{\varepsilon_k},m_{\varepsilon_k}) \to \mathcal{F}_0(u_0,m_0)$.
Propositions~\ref{prop:f_to_g} and \ref{prop:compactness} imply that 
$u_{\varepsilon_k}\weakto u_0$ weakly in $W^{1,2}_w(\Omega;\bbR^d)$,
$\chi_{y_{\varepsilon_k (\Omega)}} m_{\varepsilon_k} \to \chi_{\Omega} m_0$ strongly in $L^{2}(\bbR^d;\bbR^d)$,
and $\chi_{y_{\varepsilon_k (\Omega)}} \nabla m_{\varepsilon_k} \weakto \chi_{\Omega}\nabla m_0$ weakly in $L^{2}(\bbR^d;\bbR^d)$,
and $\mathcal{G}_{\varepsilon_k} (u_{\varepsilon_k},m_{\varepsilon_k}) \to \mathcal{G}(u_0,m_0)$.
\end{proof}

\begin{proof}[Proof of Theorem~\ref{thm:convergence_of_minimizers}]
The convergence of the sequence $u_{\varepsilon}$ can be improved to a strong convergence in $W^{1, 2} (\Om; \R^{d})$ arguing as in \cite[Subsection 7.2]{FriJamMul2006} (see also \cite[Subsection 3.2]{Bresciani-Kruzik}).
\end{proof}

\section*{Acknowledgments}
The work of S.A.\ was funded by the Austrian Science Fund (FWF) through the projects P 35359-N and ESP 65, by the Italian Ministry of Education and Research through the PRIN 2022 project ``Variational Analysis of Complex Systems in Material Science, Physics and Biology'' No.~2022HKBF5C, and by the University of Naples Federico II through the FRA project ``Regularity and Singularity in Analysis, PDEs, and Applied Sciences''. S.A.\ also acknowledges the hospitality of the University of Vienna and the TU Wien, where part of this work has been done.
M.K. was supported by the GA\v{C}R-FWF project 21-06569K, GA\v{C}R project 23-04766S, and by the Ministry of Education, Youth and Sport of the Czech Republic through the bilateral project  8J22AT017. 
A.M.\ was supported by the European Unions Horizon 2020 research and innovation programme under the Marie Sk\l{}adowska-Curie grant agreement No.~847693.
S.A.\ and A.M.\ acknowledge the kind hospitality of the ESI Institute (Vienna) during the workshop ``New perspective on Shape and Topology Optimization''.
  
Finally, the authors would like to express their gratitude to the anonymous reviewers for careful reading and their comments and suggestions that helped to improve the manuscript.


\end{document}